\documentclass[11pt,letterpaper,reqno,abstract=true]{scrartcl}

\usepackage{amsmath}
\usepackage{amsfonts}
\usepackage{amssymb}
\usepackage{amsthm}
\usepackage{amscd}

\usepackage{tikz}
\usepackage{tikz-cd}
\usetikzlibrary{cd}

\usepackage[colorlinks=true]{hyperref}
\hypersetup{urlcolor=red, citecolor=blue}

\usepackage{latexsym}
\usepackage{mathrsfs}
\usepackage{psfrag}
\usepackage[dvips]{epsfig}
\usepackage{epsfig}
\usepackage[all]{xy}         
\usepackage{wrapfig}

\usepackage[margin=2.5cm]{geometry}

\theoremstyle{plain}                              
\newtheorem{thm}{Theorem}[section]
\newtheorem{prop}[thm]{Proposition}

\newtheorem{cor}[thm]{Corollary}

\newtheorem*{mainthm}{Theorem}               

\theoremstyle{definition}                         

\theoremstyle{remark}                             

\newcommand{\R}{\mathbb{R}}                     

\newcommand{\ms}[1]{\mathscr{#1}}               

\providecommand{\norm}[1]{\left\lVert#1\right\rVert}       

\providecommand{\abs}[1]{\left\lvert#1\right\rvert}        

\begin{document}

\title{On the existence of global cross sections to 
  volume-preserving flows}

\author{Slobodan N. Simi\'c}

\maketitle


\begin{abstract}
  We establish a new criterion for the existence of a global cross
  section to a non-singular volume-preserving flow $\Phi$ on a closed
  smooth manifold $M$. Namely, if $X$ is the infinitesimal generator
  of the flow and $\Phi$ preserves a smooth volume form $\Omega$, then
  $\Phi$ admits a global cross section if there exists a smooth
  Riemannian metric $g$ on $M$ with Riemannian volume $\Omega$ and
  $g(X,X) = 1$ such that $\norm{\delta_g (i_X \Omega)}_g < 1$, where
  $\delta_g$ denotes the codifferential relative to $g$;
  (equivalently, $\norm{dX^\flat}_g < 1$). In that case, there in fact
  exists another smooth Riemannian metric on $M$ with respect to which
  the canonical form $i_X \Omega$ is co-closed and therefore harmonic.
\end{abstract}





\date{\today}

\section{Introduction and statement of the result}
\label{sec:intr-stat-result}

In this short note we consider a smooth (meaning $C^\infty$) non-singular
vector field $X$ with the associated flow $\Phi = \{ \phi_t \}$ on a
closed (i.e., compact and without boundary) smooth manifold $M$ of
dimension $n$. One of the fundamental questions one can ask about
$\Phi$ is whether it admits a global cross section, since in that case
its dynamics can be reduced to the dynamics of the first-return
diffeomorphism of the cross section. Much work has been done on this
question; see for instance, \cite{fried+82, plante72, schwartz+57,
  verj+70, ghys89, sns+97, sns+section+16, simic+23}.

Recall that a closed submanifold
$\Sigma$ of $M$ is called a \textsf{global cross section} to a $\Phi$
on $M$ if it intersects every orbit of $\Phi$ transversely. It is not
hard to see that this guarantees that the orbit of every point
$x \in \Sigma$ returns to $\Sigma$, defining the \textsf{first-return}
or \textsf{Poincar\'e map} $P$ of $\Sigma$. More precisely, for each
$x \in \Sigma$ there exists a unique $\tau(x) > 0$ (called the
\textsf{first-return time}) such that $P(x) = \phi_{\tau(x)}(x)$ is in
$\Sigma$, but $\phi_t(x) \not\in \Sigma$ for every
$t \in (0,\tau(x))$.

The first-return map $P : \Sigma \to \Sigma$ is a diffeomorphism. A
global cross section thus allows us to pass from a flow to a
diffeomorphism. To recover the flow from the first-return map one uses
the construction called \textsf{suspension}; see, e.g.,
\cite{katok+95}.

If $g$ is a smooth Riemannian metric on $M$, we will denote the
corresponding co-differential operator on $M$ by $\delta_g$. On smooth
differential forms of degree $k$, $\delta_g$ is given by
$\delta_g \omega = (-1)^{n(k+1)+1} \star_g d \star_g$, where $\star_g$ is
the corresponding Hodge star operator (cf., \cite{warner+83}). Given a Riemannian metric $g$,
we will denote by $\norm{\omega}_g$ the corresponding norm on
differential forms. More precisely, if $\omega$ is a $k$-form on $M$,
then
\begin{displaymath}
  \norm{\omega}_g = \sup_{p \in M} \abs{\omega_p}_g,
\end{displaymath}
where $\abs{\omega_p}_g$ is is the operator (or $C^0$) norm of
$\omega_p$ as a $k$-linear map
$T_p M \times \cdots \times T_p M \to \R$ relative to $g$.

It turns out that the \textsf{canonical} (closed) $(n-1)$-\textsf{form} $i_X \Omega$ of
the flow (where $i_X$ denotes the contraction by $X$) is intricately
related to the existence of a global cross section. The goal of this
note is to prove the following:

\begin{mainthm}
  Let $\Phi$ be a smooth non-singular flow with infinitesimal
  generator $X$ on a smooth closed manifold $M$. Assume $\Phi$ leaves
  a smooth volume form $\Omega$ invariant. If $M$ admits a smooth
  Riemannian metric $g$ such that the Riemannian volume of $g$ is
  $\Omega$ and
  \begin{displaymath}
    \norm{\delta_g (i_X \Omega)}_g < m_g(X)^2,
  \end{displaymath}
  where $m_g(X) = \inf_{p \in M} \norm{X}_g$ (equivalently,
  $\norm{d X^\flat}_g < m_g(X)^2$), then $\Phi$ has a global cross
  section. Furthermore, $M$ is a bundle over $S^1$.
\end{mainthm}

Note that the statement in the abstract is a direct consequence of
this more general statement. The two statements are in fact
equivalent.

\begin{cor}
  Let $N$ be a compact Riemannian manifold. Denote by $M = SN$ its
  unit tangent bundle and by $X$ the infinitesimal generator of the
  geodesic flow on $M$. Then for any smooth Riemannian metric $g$ on
  $M$ with $g(X,X) = 1$ such that the volume form of $g$ is the
  Liouville volume form $\Omega$, we have
  $\norm{\delta_g (i_X \Omega)}_g \geq 1$. Equivalently,
  $\norm{dX^\flat}_g \geq 1$.
\end{cor}

\begin{proof}
  Direct consequence of the fact that geodesic flows do not admit a
  global cross section.
\end{proof}

Thus if the codifferential of $i_X\Omega$ is small enough relative to
the size of $X$, then a global cross section exists. In the special
case when $\delta_g (i_X \Omega) = 0$, the form $i_X \Omega$ is both
closed and co-closed, hence harmonic with respect to $g$. In light of
the following result (which will be needed in the proof), it would
follow that the flow admits a global cross section.

\begin{thm}[\cite{simic+23}]    \label{thm:harmonic}
  Let $\Phi$ be a non-singular smooth flow on a closed smooth manifold
  $M$. Denote the infinitesimal generator of $\Phi$ by $X$ and assume
  that $\Phi$ preserves a smooth volume form $\Omega$. Then $\Phi$
  admits a smooth global cross section if and only if $i_X \Omega$ is
  intrinsically harmonic.
\end{thm}

A smooth differential form $\omega$ on $M$ is called
\textsf{intrinsically harmonic} if there exists a smooth Riemannian
metric $g$ on $M$ such that $\omega$ is $g$-harmonic, i.e.,
$\Delta_g \omega = 0$, where $\Delta_g = d \delta_g + \delta_g d$
denotes the Laplace-Beltrami operator on differential forms induced by
$g$ (cf., \cite{deRham+84, hodge+41, warner+83, jost+rgga+2008}). A
smooth form $\omega$ is $g$-harmonic if and only if $\omega$ is both
closed ($d\omega = 0$) and co-closed ($\delta_g \omega = 0$, i.e.,
$\star_g \omega$ is closed).

The following corollary is a straightforward consequence of the main
theorem and Theorem~\ref{thm:harmonic}.

\begin{cor}
  If $i_X \Omega$ satisfies the assumption in the theorem above, then
  there exists a smooth Riemannian metric $h$ such that $i_X \Omega$
  is co-closed with respect to $h$, hence $h$-harmonic.
\end{cor}

In his PhD thesis~\cite{honda+97} Honda gave a characterization of
intrinsically harmonic closed $(n-1)$-forms. For the case of
\emph{non-singular} (i.e., non-vanishing) $(n-1)$-forms his result
asserts the following:

\begin{thm}[\cite{honda+97}]    \label{thm:honda}
  A smooth closed non-vanishing $(n-1)$-form $\Theta$ on an $n$-dimensional
  manifold $M$ is intrinsically harmonic if and only if for every $p
  \in M$ there exists an $(n-1)$-dimensional submanifold $N_p$ of $M$
  such that the restriction of $\Theta$ to $N_p$ is a volume form for $N_p$.
\end{thm}

\subsection{Distance to the space of closed forms}

We will also need a technical result, proved in \cite{sns+section+16},
which we state below. For the sake of completeness, we also include
the proof from \cite{sns+section+16}.

Consider first a differential form $\omega$ on $M \times [0,1]$, where
$M$ is a closed manifold, and we denote the coordinate on $[0,1]$ by
$t$. Let $\pi_M : M \times [0,1] \to M$ and
$\pi_I : M \times [0,1] \to [0,1]$ be the obvious projections. Since
\begin{displaymath}
  T_{(p,t)} (M \times [0,1]) = T_p M \oplus T_t [0,1],
\end{displaymath}
any differential $k$-form on $M \times [0,1]$ can be uniquely written
as
\begin{displaymath}
  \omega = \omega_0 + dt \wedge \eta,
\end{displaymath}
where $\omega_0(v_1,\ldots,v_k) = 0$ if some $v_i$ is in the kernel of
$(\pi_M)_\ast$ and $\eta$ is a $(k-1)$-form with the analogous
property (i.e., $i_v \omega_0 = i_v \eta = 0$, for every ``vertical''
vector $v \in T(M \times [0,1])$, where $i_v$ denotes the contraction
by $v$).

Define a $(k-1)$-form $\ms{H}(\xi)$ on $M$ by
\begin{displaymath}
  \ms{H}(\omega)_p = \int_0^1 j_t^\ast \eta_{(p,t)} \: dt,
\end{displaymath}
where $j_t : M \to M \times [0,1]$ is defined by $j_t(p) =
(p,t)$. It is well-known (cf., \cite{spivak:cidg+05}) that
\begin{equation}   \label{eq:poincare}
  j_1^\ast \omega - j_0^\ast \omega = d(\ms{H}\omega) + \ms{H}(d\omega).
\end{equation}
Considering $\ms{H}$ as a linear operator from $k$-forms on
$M \times [0,1]$ to $(k-1)$-forms on $M$, both equipped with the $C^0$
norm, it is not hard to see that
\begin{equation}   \label{eq:norm}
  \norm{\ms{H}} = 1.
\end{equation}
We claim:

\begin{prop}  \label{prop:forms}
  Let $\xi$ be a $C^r$ differential $k$-form $(r \geq 1)$ on a closed
  manifold $M$. Then
  \begin{displaymath}
    \inf \{ \norm{\xi - \eta} : \eta \in C^r, \ d\eta = 0\} \leq \norm{d\xi}.
  \end{displaymath}
\end{prop}

In other words, $\norm{d\xi}$ is an upper bound on the distance from
$\xi$ to the space of closed forms. The inequality also holds for
continuous forms which admit a continuous exterior derivative.

\begin{proof}[Proof of the Proposition]
  First, let us show that the result holds on any manifold $M$ which
  is smoothly contractible to a point $p_0$ via $H : M \times [0,1]
  \to M$, where $H(p,0) = p_0$ and $H(p,1) = p$, for all $p \in
  M$. Since $H \circ j_1$ is the identity map of $M$ and $H \circ j_0$
  is the constant map $p_0$, it follows that
  \begin{displaymath}
    \xi = (H \circ j_1)^\ast \xi = j_1^\ast (H^\ast \xi) \quad
    \text{and} \quad
    0 = (H \circ j_0)^\ast \xi = j_0^\ast (H^\ast \xi).
  \end{displaymath}
  Applying \eqref{eq:poincare} to $H^\ast \xi$, we obtain
  \begin{align*}
    \xi & = \xi - 0 \\
      & = j_1^\ast (H^\ast \xi) - j_0^\ast (H^\ast \xi) \\
      & = d \ms{H}(\xi) + \ms{H}(d \xi).
  \end{align*}
  Using \eqref{eq:norm}, we obtain
  \begin{displaymath}
    \norm{\xi - d \ms{H}(\xi)} = \norm{\ms{H}(d\xi)} \leq \norm{d\xi}.
  \end{displaymath}
  Therefore, the statement of the theorem holds for contractible $M$.

  Let $M$ now be any closed manifold and $\xi$ a $C^r$ $k$-form on
  $M$, $r \geq 1$. Cover $M$ by contractible open sets $U_1, \ldots,
  U_m$. Denote the operator $\ms{H}$ restricted to forms on $U_i$ by
  $\ms{H}_i$ and let $\xi_i$ be the restriction of $\xi$ to
  $U_i$. Define a $k$-form $\eta$ on $M$ by requiring that the
  restriction of $\eta$ to $U_i$ be equal to $d \ms{H}_i (\xi_i)$. We
  claim that $\eta$ is well-defined and closed.

  Indeed, $\xi_i = \xi_j$ on $U_i \cap U_j$, and $\ms{H}_i (\omega) =
  \ms{H}_j(\omega)$ for every $k$-form $\omega$ defined on $(U_i \cap
  U_j) \times [0,1]$. Thus on $U_i \cap U_j$, we have
  $d\ms{H}_i(\xi_i) = d\ms{H}_j(\xi_j)$, so $\eta$ is
  well-defined. Since $\eta$ is locally exact, it is
  closed. It follows that
  \begin{displaymath}
    \norm{\xi - \eta} \leq \max_{1 \leq i \leq m} \norm{\xi_i -
      d\ms{H}_i(\xi_i)} \leq \max_{1 \leq i \leq m} \norm{d\xi_i} \leq \norm{d\xi}.
  \end{displaymath}
  This completes the proof of the proposition.
\end{proof}

\section{Proof of the main result}
\label{sec:proof-main-result}

Assume there exists a smooth Riemannian metric $g$ on $M$ whose volume
form equals $\Omega$ and such that
$\norm{\delta_g (i_X \Omega)}_g < m_g(X)^2$. Define a 1-form
$\theta_X$ on $M$ by $\theta_X(v) = g(X,v)$; that is,
$\theta_X = X^\flat$. Since $g$ and $X$ are smooth, so is
$\theta_X$. It is not hard to show (cf., \cite{lee+smooth+13}, Problem
16-21) that
\begin{displaymath}
  \star_g(i_X \Omega) = (-1)^{n-1} \theta_X.
\end{displaymath}
Since $\star_g \star_g = (-1)^{k(n-k)}$ on $k$-forms, we have
$\star_g \theta_X = i_X \Omega$. Recalling that the Hodge star
operator $\star_g$ is a pointwise isometry, we obtain
\begin{align*}
  \norm{d \theta_X}_g & = \norm{\star_g (d \theta_X)}_g \\
                      & = \norm{\star_g d \star_g (i_X \Omega)}_g \\
                      & = \norm{\delta_g (i_X \Omega)}_g \\
                      & < m_g(X)^2.
\end{align*}
By Proposition~\ref{prop:forms} there exists a smooth closed 1-form $\omega$
such that
\begin{displaymath}
  \norm{\theta_X - \omega}_g \leq \norm{d \theta_X}_g < m_g(X)^2.
\end{displaymath}
We claim that $\omega(X) > 0$ everywhere. Indeed:
\begin{align*}
  \omega(X) & = [\omega(X) - \theta_X(X)] + \theta_X(X) \\
            & \geq -\norm{\omega(X) - \theta_X(X)}_g + m_g(X)^2 \\
            & > 0.
\end{align*}
It follows that the smooth plane field $E = \text{Ker}(\omega)$ is of
constant codimension one and is uniquely integrable, by Frobenius's
theorem. Since $\omega(X) > 0$, $E$ is also everywhere transverse to
$X$.

Let $p \in M$ be arbitrary and denote the integral manifold of $E$
through $p$ by $N_p$. Take any $q \in N_p$. Since $T_q N_p = E_q$ is
transverse to $X_q$, the restriction of $i_X \Omega$ to $T_q N_p$ is
non-zero, i.e., a volume form. By Theorem~\ref{thm:honda},
$i_X \Omega$ is intrinsically harmonic. Therefore, $\Phi$ admits a
global cross section by Theorem~\ref{thm:harmonic}. By Theorem 2.10 in
\cite{plante72}, $M$ is a bundle over $S^1$. \qed


\bibliographystyle{amsalpha}

\bibliography{main.bib}

@article{simic+23,
	author = {Slobodan N. Simi\'{c}},
	journal = {Dynamical Systems},
	number = {2},
	pages = {314-319},
	title = {Cross-sections to flows and intrinsically harmonic forms},
	volume = {38},
	year = {2023}}

@book{hodge+41,
	author = {W. V. D. Hodge},
	publisher = {Cambridge {U}niversity {P}ress},
	title = {The {T}heory and {A}pplications of {H}armonic {I}ntegrals},
	year = {1941}}

@book{jost+rgga+2008,
	author = {J\"urgen Jost},
	edition = {fifth},
	publisher = {Springer},
	series = {Universitext},
	title = {Riemannian {G}eometry and {G}eometric {A}nalysis},
	year = {2008}}

@phdthesis{honda+97,
	author = {Ko Honda},
	school = {Princeton University},
	title = {On harmonic forms for generic metrics},
	year = {1997}}

@book{deRham+84,
	author = {Georges de Rham},
	publisher = {Springer-Verlag},
	series = {Grundlehren der mathematischen {W}issenschaften},
	title = {Differentiable {M}anifolds: {F}orms, {C}urrents, {H}armonic forms},
	volume = {266},
	year = {1984}}

@book{lee+smooth+13,
	address = {New York},
	author = {John M. Lee},
	edition = {second},
	publisher = {Springer},
	series = {Grad. Text in Math.},
	title = {Introduction to {S}mooth {M}anifolds},
	volume = {218},
	year = {2013}}

@article{sns+section+16,
	author = {Slobodan N. Simi\'{c}},
	howpublished = {{\tt arXiv:1403.2672} [math.DS].},
	journal = {Ergodic Theory Dynam. Systems},
	number = {8},
	pages = {2661--2674},
	title = {Global cross sections for {A}nosov flows},
	volume = {36},
	year = {2016}}

@book{spivak:cidg+05,
	author = {Michael Spivak},
	edition = {3rd},
	publisher = {Publish or {P}erish},
	title = {A {C}omprehensive {I}ntroduction to {D}ifferential {G}eometry},
	volume = {1},
	year = {2005}}

@article{sns+97,
	author = {Slobodan N. Simi\'c},
	journal = {Ergodic Theory Dynam. Systems},
	pages = {1211-1231},
	title = {Codimension one {A}nosov flows and a conjecture of {V}erjovsky},
	volume = {17},
	year = {1997}}

@article{fried+82,
	author = {David Fried},
	journal = {Topology},
	pages = {353-371},
	title = {The geometry of cross sections to flows},
	volume = {21},
	year = {1982}}

@inbook{ghys89,
	author = {\'{E}tienne Ghys},
	pages = {59-72},
	publisher = {Springer Verlag},
	series = {Lecture {N}otes in {M}athematics},
	title = {Codimension one {A}nosov flows and suspensions},
	volume = {1331},
	year = {1989}}

@article{plante72,
	author = {Joseph Plante},
	journal = {Amer. J. of Math.},
	pages = {729-754},
	title = {Anosov flows},
	volume = {94},
	year = {1972}}

@article{schwartz+57,
	author = {Saul Schwartzman},
	journal = {Annals of {M}ath.},
	pages = {270-284},
	title = {Asymptotic cycles},
	volume = {66},
	year = {1957}}

@article{verj+70,
	author = {Alberto Verjovsky},
	journal = {Proc. Nat. Acad. Sci. U.S.A.},
	mrclass = {57.50},
	mrnumber = {0268916},
	mrreviewer = {J. Sotomayor},
	pages = {1154--1156},
	title = {Flows with cross sections},
	volume = {66},
	year = {1970}}

@book{katok+95,
	author = {Anatole Katok and Boris Hasselblatt},
	publisher = {Cambridge {U}niversity {P}ress},
	series = {Encyclopedia of {M}athematics and its {A}pplications},
	title = {Introduction to the {M}odern {T}heory of {D}ynamical {S}ystems},
	volume = {54},
	year = {1995}}

@book{warner+83,
	author = {Frank W. Warner},
	number = {94},
	publisher = {Springer-Verlag},
	series = {Graduate {T}exts in {M}ath.},
	title = {Foundations of {D}ifferentiable {M}anifolds and {L}ie {G}roups},
	year = {1983}}

\end{document}